\documentclass[12pt]{article}
\usepackage{babel}
\usepackage{graphicx}
\usepackage{amsmath}
\usepackage{amsfonts}
\usepackage{amssymb}
\usepackage{amsthm}
\usepackage{hyperref}
\usepackage[T1]{fontenc}
\usepackage{color}
\usepackage{mathtools}
\usepackage{comment}
\usepackage{multicol}

\input xy
\xyoption{all}

\usepackage{tensor}

\newtheorem{thm}{Theorem}[section]
\newtheorem{cor}[thm]{Corollary}
\newtheorem{lem}[thm]{Lemma}
\newtheorem{prop}[thm]{Proposition}

\theoremstyle{definition}

\newtheorem{exe}[thm]{Example}
\newtheorem{ques}[thm]{Question}

\newtheoremstyle{remarque}{}{}{}{}{\it}{.}{\newline}{}
\theoremstyle{remarque}
\newtheorem*{rem}{Remark}

\usepackage[OT2,T1]{fontenc}
\DeclareSymbolFont{cyrletters}{OT2}{wncyr}{m}{n}
\DeclareMathSymbol{\Sha}{\mathalpha}{cyrletters}{"58}
\DeclareMathSymbol{\Brusse}{\mathalpha}{cyrletters}{"42}

\newcommand{\Z}{\mathbb{Z}}
\newcommand{\Q}{\mathbb{Q}}
\newcommand{\R}{\mathbb{R}}
\newcommand{\C}{\mathbb{C}}

\newcommand{\id}{\mathrm{id}}

\renewcommand{\hom}{\mathrm{Hom}}
\newcommand{\coker}{\mathrm{coker}}

\newcommand{\aut}{\mathrm{Aut}}

\newcommand{\out}{\mathrm{Out}}

\renewcommand{\int}{\mathrm{Int}\,}
\newcommand{\inn}{\mathrm{int}}

\newcommand{\gal}{\mathrm{Gal}}

\newcommand{\Sfrak}{\mathfrak{S}}
\newcommand{\Afrak}{\mathfrak{A}}

\newcommand{\spec}{\mathrm{Spec}\,}

\renewcommand{\H}{\mathrm{H}}

\newcommand{\gm}[1]{\mathbb{G}_{\mathrm{m},#1}}

\newcommand{\sln}{\mathrm{SL}_n}

\newcommand{\tw}{\mathrm{Tw}}

\usepackage{marvosym}

\title{Real approximation for homogeneous spaces with finite stabilizers}
\author{David Harari, Nguy\~{\^e}n M\d{a}nh Linh, and Giancarlo Lucchini Arteche}

\date{}

\begin{document}

\maketitle

\begin{abstract}
We prove some new cases of real appoximation for homogeneous spaces with finite stabilizers and describe the state of the art around this question, giving proofs that are well-known to experts but that, to our knowledge, cannot be found in the literature. Our main new result needs the latest advances in the topic of the Brauer--Manin obstruction for homogeneous spaces with supersolvable stabilizers. It states that any finite $k$-group that is split by a $2$-primary extension satisfies real approximation.\\

{\bf MSC classes:} 14G12, 14M17, 12G05.
\end{abstract}

\section{Introduction}

Let $k$ be a number field and let $\Omega_\R$ be the set of \emph{real} places of $k$. For $v\in\Omega_\R$, denote by $k_v$ the corresponding completion, which is isomorphic to $\R$. The different inclusions $k\hookrightarrow k_v$ correspond to the different embeddings of $k$ in the field of real numbers.

Let $A$ be a finite $k$-group scheme, which will be henceforth refered to as a \emph{$k$-group}. Fixing an algebraic closure $\bar k$ of $k$, we may and will interpret $A$ as an abstract finite group (which would be $A(\bar k)$) equipped with a continuous action of the Galois group $\Gamma:=\gal(\bar k/k)$. Here, continuity implies that the action factors through a finite quotient $\gal(L/k)$ for some finite Galois extension $L/k$.

Such a finite $k$-group can always be embedded into $\sln$ for $n$ big enough. Fix such an embedding and define $X$ as the quotient variety $\sln/A$, which is a homogeneous space for $\sln$. We have then a natural diagonal inclusion
\begin{equation}\label{eqn diag map X}
X(k)\hookrightarrow\prod_{v\in\Omega_\R}X(k_v),
\end{equation}
and it is natural to ask the following.
\begin{ques}[Real approximation for $X=\sln/A$]\label{question X}
Is the image of \eqref{eqn diag map X} dense when $X(k_v)$ is equipped with the natural real topology?
\end{ques}
An equivalent question, asked by Borovoi on MathOverflow (see \cite{BorovoiMO}), is the following.
\begin{ques}[Real approximation for $A$]\label{question H1}
Is the diagonal localization map in nonabelian cohomology
\begin{equation}\label{eqn diag map H1}
\H^1(k,A)\to \prod_{v\in \Omega_\R} \H^1(k_v,A),
\end{equation}
surjective?
\end{ques}
See \cite[\S1.2]{HarariBulletinSMF}, \cite[\S1]{GLA-AFPHEH}, or \cite[Lem. 3.7]{Linh-Padic} for an argument establishing the equivalence between the two questions.\\

These questions are a particular case of the property known as \emph{weak approximation} for homogeneous spaces, which considers an arbitrary finite set of places of $k$ instead of just the real places and an arbitrary homogeneous space $X=G/H$ of a linear algebraic $k$-group $G$. It is well-known that there exists an obstruction to this property, known as the \emph{Brauer--Manin obstruction}. A famous question/conjecture by Colliot-Th\'el\`ene \cite[Introduction]{ColliotBudapest} asks whether this obstruction is the only one. On the other hand, the third author proved in \cite{GLA-Brnr} that the Brauer--Manin obstruction does not interact with real places, so that the conjecture (which hasn't been settled in general) does imply a positive answer to Questions \ref{question X} and \ref{question H1}.\\

The goal of this note is to settle some cases of this question with a positive answer and describe the state of the art around this question. Some cases can be dealt with using very elementary arguments, stemming essentially from basic nonabelian cohomology and classical Number Theory. Others need the latest advances in the topic of the Brauer--Manin obstruction for homogeneous spaces, in particular the powerful results by Harpaz and Wittenberg in \cite{HW}. Our main result (Theorem \ref{thm: main}) states that any $k$-group that is split by a $2$-primary extension satisfies real approximation. We also give a result for certain antisolvable groups (Theorem \ref{thm: Antisolvable}).

It is worth noting that Julian Demeio recently announced a proof of Colliot-Th\'el\`ene's conjecture for all homogeneous spaces $SL_n/H$ with $H$ a solvable $k$-group \cite[Theorem 1.1]{Demeio2026Solvable}, which would imply a positive answer to Question \ref{question H1} for all solvable $A$, with no assumption on the extension splitting the group. Such a result would also shorten some of the proofs here below, but we make no use of it in this note.

\paragraph*{Acknowledgements.} We thank Mikhail Borovoi for his useful comments. In particular for pointing out the origin of the reference to Serre for real approximation for tori.

\section{Classical, well-known results}
As a warm-up, let us establish first some basic cases of this question.

\subsection{The group $A$ is constant and $|\Omega_\R|=1$}
Recall that that the $k$-group $A$ is said to be constant if the action of $\Gamma$ on $A$ is trivial. Then the map from Question \ref{question H1} becomes
\[(\hom(\Gamma,A)/\sim)\to \prod_{v\in\Omega_\R}(\hom(\Gamma_v,A)/\sim),\]
where $\Gamma_v:=\gal(\bar k_v/k_v)\simeq\Z/2\Z$ and $\sim$ relates two (continuous) homomorphisms $f,g$ if one can obtain $g$ from $f$ after conjugation by an element of $A$.

If $|\Omega_\R|=1$, which is the case for instance if $k=\Q$, we see that any homomorphism on the right-hand has image a subgroup of order $\le 2$ in $A$, so we may restrict the question to (abelian!) groups of order two and the arrow above becomes
\[k^\times/(k^\times)^2\simeq \hom(\Gamma,\Z/2\Z)\to\hom(\gal(\bar k_v/k_v),\Z/2\Z)\simeq k_v^\times/(k_v^\times)^2.\]
This map is clearly surjective, as there is only one nontrivial element on the right hand side and it can be obtained by any imaginary (i.e.~non real) quadratic extension of $k$.

\subsection{The group $A$ is abelian}
Here the answer is also positive and well-known to experts. It can be easily deduced from work by Sansuc \cite{Sansuc81}.

\begin{thm}
    If $A$ is abelian, then real approximation holds for $A$, i.e.~the map \eqref{eqn diag map H1} is surjective.
\end{thm}

We give two proofs of this result, using different techniques and classical results.

\begin{proof}[First proof]
Denote by $\Omega_k$ the set of all non-complex places of $k$ (we may ignore complex places in all that follows) and by $\Omega_f$ the set of \emph{finite} places, so that $\Omega_k=\Omega_f\sqcup\Omega_\R$.

Recall that the Tate-Shafarevich group of $A$ is defined as
\[\Sha^1(k,A):=\ker\left(\H^1(k,A)\to\prod_{v\in\Omega_k} \H^1(k_v,A)\right).\]
We will also need the following a priori bigger group
\[\Sha^1(k,\Omega_f,A):=\ker\left(\H^1(k,A)\to\prod_{v\in\Omega_f} \H^1(k_v,A)\right).\]
The classical theorems of local and global duality by Poitou and Tate (see for instance \cite[8.6.10]{NSW}) yield an exact sequence (see \cite[Lem.~9.2.2]{NSW})
\[0\to\Sha^1(k,\widehat A)\to\Sha^1(k,\Omega_f,\widehat A)\to\coker^1(k,\Omega_\R,A)^\vee,\]
where $\coker^1(k,\Omega_\R,A)$ is the cokernel of the map \eqref{eqn diag map H1} (which is a group homomorphism since $A$ is abelian) and the notation $(\cdot)^\vee$ denotes $\hom(\cdot,\Q/\Z)$. Since we intend to prove that this cokernel is trivial, it will suffice to prove that $\Sha^1(k,\Omega_f,\widehat A)=\Sha^1(k,\widehat A)$. In other words, if a class $\alpha\in \H^1(k,\widehat A)$ is trivial over every finite place, then it must be trivial over every real place. But this is a simple consequence of Chebotarev's density theorem, as it was already noticed by Sansuc in \cite[Lem.~1.1.(iii)]{Sansuc81}.
\end{proof}

The second proof uses the fact that real approximation for tori holds. This is proved by Voskresenskii in \cite[Thm.~11.5]{VoskresenskiiToresII}, although he attributes the result to Serre (unpublished). Sansuc confirms this claim in \cite[Cor.~3.5]{Sansuc81}.

\begin{proof}[Second proof]
Let $L/k$ be a Galois extension such that the $\Gamma$-action factors through $\Gamma_{L/k}:=\gal(L/k)$. We may and will assume that $L$ contains the $n$-th roots of unity for $n=|A|$.

Consider the Cartier dual $\widehat A:=\hom(A,\gm)$. It is a finite $k$-group as well, hence it can be seen as an abstract group (non-canonically) isomorphic to $A$, equipped with an action of $\Gamma$ that factors through $\Gamma_{L/k}$. We can always write $\widehat A$ as a quotient by an induced $\Gamma_{L/k}$-module, so that we have an exact sequence of $\Gamma_{L/k}$-modules (and hence of $\Gamma$-modules)
\[0\to \widehat T\to \widehat P\to\widehat A\to 0.\]
Taking the Cartier dual once again, we recover an exact sequence
\[1\to A\to P\to T\to 1,\]
where $P,T$ are algebraic $k$-tori and $P$ is \emph{quasitrivial}. An application of Shapiro's lemma and Hilbert's Theorem 90 tells us that $\H^1(K,P)=0$ for any overfield $K/k$ (we say then that $P$ is \emph{$k$-special}). Then we get a commutative square with surjective horizontal arrows
\[\xymatrix{
    T(k) \ar@{->>}[r]^-{\delta} \ar[d] & \H^1(k,A) \ar[d] \\
    \prod_{v\in\Omega_\R}T(k_v) \ar@{->>}[r]^-{\delta} & \prod_{v\in\Omega_\R} \H^1(k_v,A),
}
\]
where the right vertical arrow is the map \eqref{eqn diag map H1} and the horizonal arrows are the usual connecting maps in cohomology. These maps are locally constant for topological fields, hence the lower horizontal map is continuous for the discrete topology on the right. Thus, if one proves that the image of the left vertical arrow is dense, surjectivity of the right vertical arrow follows. But this last assertion corresponds to real approximation for the torus $T$. This concludes the proof.
\end{proof}

\begin{rem}
Note that in this last proof we constructed the torus $T$ as a quotient $P/A$. The argument used above, based on the fact that $P$ is $k$-special, is essentially the one that proves the equivalence of Questions \ref{question X} and \ref{question H1} if we replace $P$ by $\sln$, which is also a $k$-special group.
\end{rem}

\section{Some useful elementary results}
The following propositions will be used in what follows, but we present them in a separate section as they are very elementary.

\begin{prop}\label{prop: reduction aux 2-Sylow}
Assume that $A$ admits a $\Gamma$-invariant $2$-Sylow subgroup $S$. If $S$ satisfies real approximation, then so does $A$.
\end{prop}

\begin{rem}
The hypothesis on real approximation for $S$ is superfluous, assuming the result claimed by Demeio that we mentioned in the introduction \cite[Theorem 1.1]{Demeio2026Solvable}.
\end{rem}

\begin{proof}
We have the commutative square
\[\xymatrix{
\H^1(k,S) \ar[r] \ar[d] & \H^1(k,A) \ar[d] \\
\prod_{v\in\Omega_\R} \H^1(k_v,S) \ar[r] & \prod_{v\in\Omega_\R} \H^1(k_v,A),
}\]
and we are assuming that the vertical arrow on the left is surjective. It will suffice then to prove that the lower horizontal arrow is surjective, which we can do place by place.

Fix then a completion $k_v=\R$ and consider then a class $\alpha\in \H^1(\R,A)$, given by a cocycle $a:\Gamma_{\R}\to A$. By \cite[Chap.~I, \S{5.4}, Prop.~37]{Serre1994Galois}, we know that $\alpha$ is in the image of $\H^1(\R,S)$ if and only if the twisted homogeneous space $_{a}(A/S)$ has an $\R$-point. But this is an \'etale $\R$-scheme of odd degree, that is, it corresponds to $\spec(B)$ for some finite, separable $\R$-algebra $B$ of odd dimension over $\R$. Such an algebra is simply a product of fields isomorphic to either $\C$ or $\R$. Since the degree is odd, there must be at least one copy of $\R$, which implies that the scheme has an $\R$-point. This concludes the proof.
\end{proof}

The following result gives a condition under which such Sylow subgroups can exist.

\begin{prop}\label{prop: il existe un 2-Sylow}
Let $p$ be a prime number and $A$ be a $k$-group. Assume that the action of $\Gamma$ on $A$ factors through a finite quotient which is a $p$-group. Then there exists a $\Gamma$-stable $p$-Sylow subgroup $S$ of $A$
\end{prop}

\begin{proof}
    Let $\Gamma'$ be a finite quotient of $\Gamma$ that acts on $A$ and is a $p$-group. Then this quotient acts on the set of $p$-Sylow subgroups of $A$. Since $\Gamma'$ is a $p$-group, every orbit of this action must have a cardinal that is a power of $p$. On the other hand, since the number of Sylow subgroups is coprime to $p$, the classical orbit formula applied to this action tells us that at least one of these powers must be $1$, which yields an invariant Sylow subgroup.
\end{proof}

Recall that a $k$-group $A$ is said to be \emph{supersolvable} if there exists a chain of normal $k$-subgroups (i.e. $\Gamma$-stable subgroups) of $A$
\[\{1\}=A_0\subseteq A_1\subseteq \cdots\subseteq A_{n-1}\subseteq A_n=A,\]
such that $A_i/A_{i-1}$ is cyclic for every $1\leq i\leq n$.

When we restrict the definition to constant groups, supersolvable groups include nilpotent groups and dihedral groups, but not every solvable group. For instance, the alternating group $A_4$ is not supersolvable, even with a trivial $\Gamma$-action.

Note also that not every abelian $k$-group is supersolvable. Indeed, it suffices to consider the abstract group $A=(\Z/2\Z)^2$ with an action of $\Gamma$ via a cyclic quotient of order 3 that permutes the three nontrivial elements in $A$. Then the $k$-group does not admit any nontrivial $\Gamma$-stable (normal) subgroup, so it is impossible to obtain cyclic quotients.

The following result provides a family of non-constant supersolvable groups.

\begin{prop}\label{prop: 2-2-groupe hyperresoluble}
Let $p$ be a prime number and $A$ be a $k$-group of $p$-primary order. Assume that the action of $\Gamma$ on $A$ factors through a finite quotient that is also a $p$-group. Then the $k$-group $A$ is supersolvable.
\end{prop}

\begin{proof}
Assume that $A$ admits a central, $\Gamma$-stable cyclic subgroup $C$ and consider the quotient $A/C$. We claim that if $A/C$ is supersolvable, then so is $A$. Indeed, by definition, there exists normal subgroups
\[\{1\}=B_0\subseteq B_1\subseteq \cdots\subseteq B_{n-1}\subseteq B_n=A/C,\]
such that $B_i/B_{i-1}$ is cyclic for every $1\leq i\leq n$. Write $\pi:A\to A/C$ the canonical projection and define
\begin{itemize}
    \item $A_0:=\{1\}$;
    \item $A_i:=\pi^{-1}(B_{i-1})$ for $1\leq i\leq n+1$.
\end{itemize}
Then each $A_i$ is $\Gamma$-invariant and is normal in $A$ as it is the preimage of a normal subgroup of $A/C$ via a $\Gamma$-equivariant morphism. Moreover, we have $A_1=C$, $A_{n+1}=A$ and we have inclusions
\[\{1\}=A_0\subseteq A_1\subseteq \cdots\subseteq A_{n}\subseteq A_{n+1}=A,\]
where $A_i/A_{i-1}=B_{i-1}/B_{i-2}$ is cyclic for every $2\leq i\leq n$, while $A_1/A_0=C$, which is also cyclic. This proves that $A$ is supersolvable.\\

By induction on the cardinal of $A$ and since a trivial group is obviously supersolvable, we are reduced to proving that, when $A$ satisfies the hypotheses of the proposition, there exists a central, $\Gamma$-stable cyclic subgroup. 

Now, since $A$ is a $p$-group, its center $Z$ is nontrivial and characteristic in $A$, hence $\Gamma$-stable. As in the previous proof, let $\Gamma'$ be a finite quotient of $\Gamma$ that acts on $A$ and is a $p$-group. Then this quotient acts on $Z$ and hence on the set of \emph{nonzero} elements in $Z$. But the cardinal of this set is coprime to $p$, so the same classical argument tells us that this action has at least one invariant element. The cyclic subgroup $C$ generated by this element is then normal (since it is central) and clearly $\Gamma$-stable.
\end{proof}

\section{More advanced results}

\subsection{Groups split by $2$-extensions}

We present now the main result of this note, which depends on the elementary results from the previous section and on a deep result by Harpaz and Wittenberg in \cite{HW}.

\begin{thm}\label{thm: main}
    Let $A$ be a finite $k$-group on which the action of $\Gamma$ factors through a finite quotient which is a $2$-group. Then $A$ satisfies real approximation.
\end{thm}

Note that this result includes all constant $k$-groups.

\begin{proof}
By Proposition \ref{prop: il existe un 2-Sylow}, we know that there exists a $\Gamma$-stable $2$-Sylow subgroup $S$ of $A$. By Proposition \ref{prop: reduction aux 2-Sylow}, it will suffice then to prove the result for $S$. Now, by Proposition \ref{prop: 2-2-groupe hyperresoluble}, this group is supersolvable. Then we know by \cite[Thm.~B]{HW} that the Brauer--Manin obstruction to weak approximation is the only obstruction. By \cite[Thm.~6.1]{GLA-Brnr}, this implies that $X=\sln/S$ satisfies real approximation. This concludes the proof.
\end{proof}

This is as far as we can go without imposing any conditions on the structure of the abstract group $A$. In the following sections we present some results that relax the hypothesis on the action of $\Gamma$, but at the price of imposing conditions on the group $A$.

\subsection{Reminders on nonabelian Galois cohomology} \label{subsec: recall nonabelian cohomology}

Let $A$ be a finite $k$-group. Then $\Gamma = \gal(\bar{k}/k)$ acts on $\aut(A)$ via the formula
    \begin{equation*}
        (\tensor[^s]{\phi}{})(a) := \tensor[^s]{(\phi(}{}\tensor[^{s^{-1}}]{a}{}))
    \end{equation*}
To each cocycle $\sigma: \Gamma \to \aut(A)$ (that is, $\sigma_{st} = \sigma_s \circ \tensor[^s]{\sigma}{_t}$ for all $s,t \in \Gamma$) is associated a {\em twisted $k$-form} $\tensor[_{\sigma}]{A}{}$ of $A$, on which $\Gamma$ acts via $(s,a) \mapsto \tensor[^s]{\sigma}{}(a)$. Such a $k$-form is called {\em inner} if $\sigma$ takes values in $\int(A) \cong A/Z(A)$. If furthermore $\sigma$ is the image of a cocycle (also denoted by) $\sigma:\Gamma \to A$, then we have a {\em twisting bijection}
    \begin{equation*}
        \tw_\sigma: \H^1(k,\tensor[_{\sigma}]{A}{}) \to \H^1(k,A), \quad [\tau] \mapsto [\tau\sigma],
    \end{equation*}
which takes the distinguished element of $\H^1(k,\tensor[_{\sigma}]{A}{})$ to $[\sigma]$ (see \cite[Chap.~I, \S{5.3}, Prop.~35 bis]{Serre1994Galois}). In the remainder of the paper, we shall consider centerless groups, so that any inner form of $A$ is a $k$-form twisted by an $A$-valued cocycle.

An {\em outer action} of $\Gamma$ on an abstract group $\bar{A}$ is a continuous homomorphism $\kappa:\Gamma \to \out(\bar{A}):=\aut(\bar{A})/\int(\bar{A})$. It is easily checked that two $k$-forms $A$ and $B$ of $\bar{A}$ induce the same outer Galois action if and only if $B$ is an inner $k$-form of $A$. One can then define the set $\H^2(k,\bar{A},\kappa)$ of nonabelian second Galois cohomology, {\em cf.} \cite[\S{1}]{FSS}. This is not a pointed set, but it is equipped with a distinguished subset of {\em neutral elements}. Given a short exact sequence
    \begin{equation*}
        1 \to A \to B \to C \to 1
    \end{equation*}
of finite $k$-groups, together with a cocycle $\sigma: \Gamma \to C$, one has a natural outer action $\kappa_\sigma: \Gamma \to \out(A)$ and a class $\eta_\sigma \in \H^2(k,A,\kappa_\sigma)$ which is neutral precisely when $[\sigma] \in \H^1(k,C)$ is the image of some element from $\H^1(k,B)$ (see \cite[Chap.~IV, Prop. 2.5.5 and
4.2.8]{Giraud} for a proof using the language of {\em gerbes}).

\subsection{Some centerless groups}

\begin{lem} \label{lem:InvariantUnderInnerForm}
    Let $A$ be a finite centerless $k$-group, and $B$ and inner $k$-form of $A$. Then $A$ satisfies real approximation if and only if $B$ does.
\end{lem}

\begin{proof}
    Since $A$ is centerless, one has $B = \tensor[_\sigma]{A}{}$ for some cocycle $\sigma: \Gamma \to A$. The twisting bijection $\tw_\sigma: \H^1(k,B) \to \H^1(k,A)$ fits in a commutative diagram
        \begin{equation*}
            \xymatrix{
                \H^1(k,B) \ar[r]^{\tw_\sigma} \ar[d] & \H^1(k,A) \ar[d] \\
                \prod_{v \in \Omega_\R} \H^1(k_v,B) \ar[r]^{\tw_\sigma} & \prod_{v \in \Omega_\R} \H^1(k_v,A).
            }
        \end{equation*}
   Thus, the right vertical map is surjective if and only if the left one is so.
\end{proof}

\begin{cor} \label{cor:InvariantUnderInnerForm}
    Let $A$ be a finite centerless $k$-group. Assume that its Galois action is inner. Then $A$ satisfies real approximation.
\end{cor}

\begin{proof}
    This is because finite constant $k$-groups satisfy real approximation by Theorem \ref{thm: main}.
\end{proof}

A finite group $A$ is said to be {\em almost complete} if it is centerless and the extension
    \begin{equation*}
        1 \to A \to \aut(A) \to \out(A) \to 1
    \end{equation*}
is split. Almost complete finite simple groups were fully characterized in \cite{LMM}. These include the alternating group $\Afrak_n$ for $n = 5$ or $n \ge 7$ and all 26 sporadic groups.

\begin{prop}\label{prop almost complete}
Let $A$ be a finite, almost complete $k$-group.  Assume that the outer action of $\Gamma$ on $A$ factors through a finite quotient which is a $2$-group. Then $A$ satisfies real approximation.
\end{prop}

\begin{proof}
Let $\kappa: \Gamma \to \out(A)$ be the outer Galois action obtained from the given Galois action of $\Gamma$ on $A$. Since $A$ is almost complete, we have a section $\sigma:\out(A)\to\aut(A)$. The action $\sigma \circ \kappa$ then yields an inner $k$-form $A'$ of $A$ such that the action of $\Gamma$ on $A'$ factors through a finite quotient which is a $2$-group. By Theorem \ref{thm: main}, real approximation holds for $A'$. Since $A$ is centerless, Lemma \ref{lem:InvariantUnderInnerForm} tells us that $A$ also satisfies real approximation.
\end{proof}

Recall that the simple factors of a finite group are the simple groups arising from the Jordan--H\"older decomposition. In the next subsection, we shall extend Proposition \ref{prop almost complete} to the class of antisolvable groups with almost complete simple factors. The proof is by induction, the base case being that for {\em characteristically simple} groups ({\em i.e.} groups whose only characteristic subgroups are the trivial ones).

\begin{lem} \label{lem: characteristically simple}
    Let $A$ be a characteristically simple nonabelian finite $k$-group. Assume that the simple factors of $A$ are almost complete. Then $A$ is almost complete.

    In particular, if we assume moreover that the outer action of $\Gamma$ on $A$ factors through a finite quotient which is a $2$-group, then real approximation holds for $A$.
\end{lem}

\begin{proof}
    By \cite[Chap.~2, Thm.~1.4]{Gorenstein}, $A = F^n$ (as an abstract group) for some finite simple (nonabelian, centerless) group $F$ and some integer $n \ge 1$. By assumption, there exists a section $\sigma: \out(F) \to \aut(F)$. Now, one has $\aut(A) = \aut(F)^n \rtimes \Sfrak_n$ by \cite[Thm.~3.1]{BCM2006Automorphism}, where the symmetric group $\Sfrak_n$ acts on $\aut(F)^n$ by permuting the coordinates. It follows that $\out(A) = \out(F)^n \rtimes \Sfrak_n$ and hence we have a section
        \begin{equation*}
            \sigma' = (\sigma \times \cdots \times \sigma) \rtimes \id_{\Sfrak_n}: \out(A) \to \aut(A).
        \end{equation*}
    This proves the first assertion. The second one follows by Proposition \ref{prop almost complete}.
\end{proof}

\subsection{Some antisolvable groups}

A finite group $G$ is said to be {\em antisolvable} if its simple factors are nonabelian. In this subsection, we prove

\begin{thm} \label{thm: Antisolvable}
    Let $A$ be a finite antisolvable $k$-group. Assume the following.

    \begin{enumerate}
            \item The simple factors of $A$ are almost complete.

            \item The outer action of $\Gamma$ on $A$ factors through a finite quotient which is a $2$-group.
        \end{enumerate}
    Then $A$ has an inner $k$-form on which the action of $\Gamma$ factors through a finite quotient which is a $2$-group. In particular, real approximation holds for $A$.
\end{thm}

The proof goes by induction on $|A|$ and the base case is the one where $A$ is characteristically simple, which has been dealt with in Lemma \ref{lem: characteristically simple}. In order to proceed with the induction step, we need the following result.

\begin{lem}\label{lem: antisolvable}
Let $A$ be a finite group that is not characteristically simple. Let $B \subseteq A$ be a maximal proper characteristic subgroup of $A$. Then the quotient $C:=B/A$ is characteristically simple.
\end{lem}

\begin{proof}
Denote by $\pi: A \to C$ the canonical projection. Any automorphism of $A$ leaves $B$ invariant, hence induces an automorphism of $C$. It follows that the preimage $\pi^{-1}(C')$ of any characteristic subgroup $C' \subseteq C$ is a characteristic subgroup of $A$ containing $B$. This implies either $\pi^{-1}(C') = A$ (in which case $C' = C$) or $\pi^{-1}(C') = B$ (in which case $C' = 1$).
\end{proof}

\begin{proof}[Proof of Theorem \ref{thm: Antisolvable}]
We proceed by induction on $|A|$. If $A$ is characteristically simple, we are done by Lemma \ref{lem: characteristically simple}. Otherwise, by Lemma \ref{lem: antisolvable} there exists a characteristic subgroup $B \subseteq A$ such that $C:=B/A$ is characteristically simple. Note that the simple factors of both $B$ and $C$ are also almost complete since together they are precisely those of $A$.

Note that $1 \to B \to A \xrightarrow{\pi} C \to 1$ is a short exact sequence of finite $k$-groups. Since $B$ is characteristic, there is a natural homomorphism $\aut(A) \to \aut(C)$, which induces a homomorphism $\out(A) \to \out(C)$, and the outer action of $\Gamma$ on $C$ is precisely the composite of this homomorphism with the outer action $\Gamma \to \out(A)$. It follows from Lemma \ref{lem: characteristically simple} that $C$ is almost complete and hence, reasoning as in the proof of Proposition \ref{prop almost complete}, we see that there is a cocycle $\sigma: \Gamma \to C$ such that $\Gamma$ acts on $\tensor[_{\sigma}]{C}{}$ through a finite quotient which is a $2$-group. As in Subsection \ref{subsec: recall nonabelian cohomology}, we obtain an outer action $\kappa_\sigma: \Gamma \to \out(B)$ together with a cohomology class $\eta_\sigma \in \H^2(k,B,\kappa_\sigma)$. Now, \cite[Thm.~3]{GLA-Antisolvable} tells us that $\eta_\sigma$ is neutral (in fact, it is the only element of $\H^2(k,B,\kappa_\sigma)$), so we may assume that $\sigma$ is the image of a cocycle $\sigma': \Gamma \to A$.

Twisting by $\sigma'$ does not change the outer Galois action. Thus, we may assume $[\sigma] = 1$. In other words, there are normal subgroups $\Gamma'$ and $\Gamma''$ of $\Gamma$, of $2$-primary index, such that the respective homomorphisms $\Gamma \to \out(A)$ and $\Gamma \to \aut(C)$ factor through $\Gamma/\Gamma'$ and $\Gamma/\Gamma''$. Let $\Delta:=\Gamma' \cap \Gamma''$, which is again a normal subgroup of $\Gamma$ of $2$-primary index. Each element $s \in \Delta$ then acts on $A$ via some inner automorphism $\inn(a_s)$ (where $a_s \in A$). As $s$ acts trivially on $C$, the image $c_s:=\pi(a_s)$ is central in $C$. Since $C$ is centerless, we have $c_s = 1$, or $a_s \in B$. This means $s$ acts on $B$ via an inner automorphism, or equivalently that the homomorphism $\Gamma \to \out(B)$ kills $s$.

We have proved that (after replacing $A$ by an inner $k$-form if needed), the outer action of $\Gamma$ on $B$ factors through a finite quotient $\Gamma/\Delta$ which happens to be a $2$-group. By induction hypothesis, there is a cocycle $\tau: \Gamma \to B$ such that $\Gamma$ acts on $\tensor[_\tau]{B}{}$ through the a finite quotient which is a $2$-group. Twisting $A$ by $\tau$ does not change the outer Galois action on $A$ and the Galois action on $C$. Therefore, we may assume the existence of a subgroup $\Delta' \subseteq \Gamma$ of $2$-primary index such that the homomorphisms $\Gamma \to \aut(B)$, $\Gamma \to \out(A)$, and $\Gamma \to \aut(C)$ factor through $\Gamma/\Delta'$. As above, any $s \in \Delta'$ acts on $A$ via an inner automorphism $\inn(b_s)$ for some $b_s \in B$. But then $b_s = 1$ since $B$ is centerless, so the {\em action} $\Gamma\to \aut(A)$ (not just the outer action $\Gamma \to \out(A)$) factors through the $2$-group $\Gamma/\Delta'$. In particular, Theorem \ref{thm: main} implies the real approximation property for $A$.
\end{proof}

\begin{exe}
The group $A_n$ for $n=5$ or $n\geq 7$ is a simple, almost complete group whose outer automorphism group is a $2$-group. This is also true for all 26 sporadic groups, so that the hypotheses of Theorem \ref{thm: Antisolvable} always hold and we may deduce that these groups satisfy real approximation regardless of the Galois action.

A slightly more general example would be any direct product $A\times B$ of two such groups. Indeed, if $A$ and $B$ are not isomorphic, then $\aut(A\times B)\simeq \aut(A)\times\aut(B)$. On the other hand, we already mentioned in the proof of Lemma \ref{lem: characteristically simple} that $\aut(A^2)\simeq \aut(A)^2\rtimes S_2$. In both cases we see then that the outer automorphism group would be a $2$-group.

In order to construct more involved examples, it is worth noting that in \cite{Bercov} Bercov characterizes anti-solvable finite groups whose composition factors are almost complete as iterated twisted wreath products of their simple factors.
\end{exe}


\begin{thebibliography}{ABCD00}

\bibitem[Ber67]{Bercov}
{\bf R.~Bercov.} On groups without abelian composition factors. {\it J.~Algebra} 5, 106--109, 1967.

\bibitem[BCM06]{BCM2006Automorphism}
{\bf J.~N.~S.~Bidwell, M.~J.~Curran, D.~J.~McCaughan.} Automorphisms of direct products of finite groups. {\it Arch.~Math.} 86, 481-–489, 2006. 

\bibitem[Bor10]{BorovoiMO} {\bf M.~Borovoi}. Real approximation for homogeneous spaces of linear algebraic groups. URL (version: 2010-12-28): \url{https://mathoverflow.net/q/50570}

\bibitem[CT03]{ColliotBudapest}
{\bf J.-L.~Colliot-Th{\'e}l{\`e}ne.} Points rationnels sur les fibrations. {\it Higher dimensional varieties and rational points ({B}udapest, 2001)}, 171--221, Bolyai Soc. Math. Stud., 12, {\it Springer, Berlin}, 2003.

\bibitem[Dem26]{Demeio2026Solvable}
{\bf J.~L.~Demeio} {\it Solvable descent and the Grunwald problem for solvable groups}, Preprint, 103 pages, 2026. Available at \href{https://arxiv.org/abs/2604.18099}{arXiv:2604.18099}.

\bibitem[FSS98]{FSS}
{\bf Y.~Z.~Flicker, C.~Scheiderer, R.~Sujatha} Grothendieck's theorem on non-abelian $H^2$ and local-global principles. {\it J.~Amer.~Math.~Soc.} 11:3, 731--750, 1998.

\bibitem[Gir71]{Giraud}
{\bf J.~Giraud.} {\it Cohomologie non ab\'elienne}. Grundlehren der Mathematischen Wissenschaften, No. {\bf 179}. Springer Berlin Heidelberg, 1971.

\bibitem[Gor80]{Gorenstein}
{\bf D.~Gorenstein.} {\it Finite groups}. AMS Chelsea Publishing Series, No. {\bf 301}. Amer. Math. Soc., second edition, 1980.

\bibitem[Har07]{HarariBulletinSMF}
{\bf D.~Harari.} Quelques propri\'et\'es d'approximation reli\'ees \`a la cohomologie galoisienne d'un groupe alg\'ebrique fini. {\it Bull.~Soc.~Math.~France} 135(4), 549--564, 2007.

\bibitem[HW20]{HW}
{\bf Y.~Harpaz, O.~Wittenberg.} Z\'ero-cycles sur les espaces homog\`enes et probl\`eme de Galois inverse. {\it J.~Amer.~Math.~Soc.} 33, 775--805, 2020.

\bibitem[Lin24]{Linh-Padic}
{\bf N.~M.~Linh.} Arithmetics of homogeneous spaces over $p$-adic function fields. {\it J.~London Math.~Soc} 109(1), e12842, 2024.

\bibitem[LMM03]{LMM}
{\bf A.~Lucchini, F.~Menegazzo, M.~Morigi}, On the existence of a complement for a finite simple group in its automorphism group, {\it Ill.~J.~Math.} 47(1-2), 395--418, 2003. 

\bibitem[LA14]{GLA-AFPHEH}
{\bf G.~Lucchini {A}rteche.} Approximation faible et principe de Hasse pour des espaces homog\`enes \`a stabilisateur fini r\'esoluble. {\it Math.~Ann.} 360, 1021--1039, 2014.

\bibitem[LA19]{GLA-Brnr}
{\bf G.~Lucchini {A}rteche.} The unramified Brauer group of homogeneous spaces with finite stabilizer. {\it Trans.~Amer.~Math.~Soc.} 372, 5393--5408, 2019.

\bibitem[LA22]{GLA-Antisolvable}
{\bf G.~Lucchini {A}rteche.} On homogeneous spaces with finite anti-solvable stabilizers. {\it C.~R.~Math.} 360, 777--780, 2022.

\bibitem[NSW08]{NSW}
{\bf J.~Neukirch, A.~Schmidt, K.~Wingberg.} {\it Cohomology of number fields}. Grundlehren der Mathematischen Wissenschaften, No. {\bf 323}.
Springer-Verlag, Berlin, second edition, 2008.

\bibitem[San81]{Sansuc81}
{\bf J.-J.~Sansuc.} Groupe de {B}rauer et arithm\'etique des groupes alg\'ebriques lin\'eaires sur un corps de nombres. {\it J.~Reine Angew.~Math.} 327, 12--80, 1981.

\bibitem[Ser94]{Serre1994Galois}
{\bf J.-P.~Serre.} {\it Cohomologie galoisienne}, Lecture Notes in Mathematics, No.~{\bf 5}. Springer, Berlin, Heidelberg, fifth edition, 1994.

\bibitem[Vos98]{VoskresenskiiToresII}
{\bf V.~E.~Voskresenski{\u{\i}}.} {\it Algebraic groups and their birational invariants.} Transl.~Math.~Monogr., No.~{\bf 179}. American Mathematical Society, Providence, RI, 1998.

\end{thebibliography}
\end{document}